\documentclass
{amsart}

\usepackage{amssymb,amscd}
\usepackage[all,line,arc,curve,color,frame,pdf]{xy}
\usepackage{tikz}
\usepackage{tikz-cd}
\usetikzlibrary{positioning, trees, snakes}
\usepackage[textsize=tiny]{todonotes}
\usepackage{hyperref}
\usepackage[shortlabels]{enumitem}
\usepackage{float}
\usepackage[paper=a4paper, margin=3.5cm]{geometry}
\usepackage{cleveref}
\usepackage{comment}
\usepackage{scalerel}
\usepackage[normalem]{ulem}
\usepackage{booktabs}
\usepackage{mathtools}
\usepackage{cite} 

\numberwithin{equation}{section}

\newtheorem{Theorem}{Theorem}[section]
 \newtheorem{Lemma}[Theorem]{Lemma}
 \newtheorem*{TheoremMain}{Theorem}
 \newtheorem{Corollary}[Theorem]{Corollary}
 \newtheorem{Proposition}[Theorem]{Proposition}
 \newtheorem{Remark}[Theorem]{Remark}

 \newtheorem{Definition}[Theorem]{Definition}
 
 \newtheorem{Conjecture}[Theorem]{Conjecture}

\DeclareMathOperator\MLdeg{ML-degree}

\usepackage[textsize=tiny]{todonotes}

\usepackage{enumitem}

\newcommand\ignore[1]{}

\newcommand\CC{{\mathbb{C}}}
\newcommand\PP{{\mathbb{P}}}

\def\operatorname#1{\mathop{\rm #1}\nolimits}

\def\rank{\operatorname{rank}}

\def\det{\operatorname{det}}

\newcommand{\pb}{\ar@{}[dr]|{\text{\pigpenfont J}}}

\newcommand{\xdasharrow}[2][->]{
\tikz[baseline=-\the\dimexpr\fontdimen22\textfont2\relax]{
\node[anchor=south,font=\scriptsize, inner ysep=1.5pt,outer xsep=2.2pt](x){#2};
\draw[shorten <=3.4pt,shorten >=3.4pt,dashed,#1](x.south west)--(x.south east);
}}

\def\opn#1#2{\def#1{\operatorname{#2}}} 
\opn\Id{Id}
\begin{document}
\title{On the Maximum Likelihood Degree for Gaussian graphical models}


\author[Am\'endola]{Carlos Am\'endola}
\address{
	Technische Universit\"at Berlin,  Institut f\"ur Mathematik,
	D-10623 Berlin, Germany
}
\email{amendola@math.tu-berlin.de}

\author[Dinu]{Rodica Andreea Dinu}
\address{%
	University of Konstanz, Fachbereich Mathematik und Statistik, Fach D 197 D-78457 Konstanz, Germany, and Simion Stoilow Institute of Mathematics of the Romanian Academy, Calea Grivitei 21, 010702, Bucharest, Romania}
\email{rodica.dinu@uni-konstanz.de}

\author[Micha{\l}ek]{Mateusz Micha{\l}ek}
\address{
	University of Konstanz, Germany, Fachbereich Mathematik und Statistik, Fach D 197
	D-78457 Konstanz, Germany
}
\email{mateusz.michalek@uni-konstanz.de}

\author[Vodi\v{c}ka]{Martin Vodi\v{c}ka}
\address{Šafárik University, Faculty of Science, Jesenná 5, 04154 Košice, Slovakia}
	\email{martin.vodicka@upjs.sk}

\begin{abstract}
In this paper we revisit the likelihood geometry of Gaussian graphical models. We give a detailed proof that the ML-degree behaves monotonically on induced subgraphs. Furthermore, we complete a missing argument that the ML-degree of the $n$-th cycle is larger than one for any $n\geq 4$, therefore completing the characterization that the only Gaussian graphical models with rational maximum likelihood estimator are the ones corresponding to chordal (decomposable) graphs. Finally, we prove that the formula for the ML-degree of a cycle conjectured by Drton, Sturmfels and Sullivant provides a correct lower bound.
\end{abstract}

\thanks{RD was supported by the Alexander von Humboldt Foundation, MM was supported by the DFG grant 467575307, MV was supported by Slovak VEGA grant 1/0152/22.}
\maketitle

\section{Introduction}


Gaussian graphical models have been used extensively in statistics to express dependence relationships between random variables \cite{lauritzen1996graphical,SethBook,StUh}. The key connection between algebra and statistics stems from the fact that for an $n$-dimensional Gaussian random vector $X \sim N(0,\Sigma)$,
\begin{equation}\label{eq:ci}
     X_i \perp\!\!\!\perp X_j | X_{[n]\setminus \{ i , j \}} \iff (\Sigma^{-1})_{ij}=0.
\end{equation}
In other words, the random variables corresponding to the $i$-th and $j$-th coordinates of $X$ ($i,j \in [n]=\{1,\dots,n\}$) are conditionally independent given the rest of the variables, if and only if the corresponding entry in the inverse of the covariance matrix is 0.

This motivates the following definition, where a graph $G=(V,E)$ is used to encode the sparsity in \eqref{eq:ci}.

\begin{Definition}
    Let $G$ be an undirected graph with $V=[n]$. The \emph{Gaussian graphical model} given by $G$ consists of all $n$-dimensional Gaussian distributions with covariance matrix $\Sigma$ such that $(\Sigma^{-1})_{ij}=0$ for every missing edge $i-j$.
\end{Definition}

In this way, the Gaussian graphical model associated to a graph $G$ is a \emph{linear concentration model}, as it imposes linear restrictions on the inverse covariance matrices. We denote this linear space by $L_G$, and denote the closure of the set of matrices obtained by inverting the
regular matrices in $L_G$ by $L_G^{-1}$.

Set $S^2(k^n)$ to be the space of $n \times n$ symmetric matrices over the
field $k$. Given a graph $G$ and sample data from $X$ summarized in a sample covariance matrix $S \in S^2(\mathbb{R}^n)$, one wishes to estimate the covariance matrix $\Sigma$. A natural estimator is the \emph{maximum likelihood estimator (MLE)}, which maximizes the \emph{log-likelihood function}
\begin{equation}\label{eq:lik}
    \ell(\Sigma) = \log \det (\Sigma^{-1}) - \mathrm{tr}(S\Sigma^{-1}).
\end{equation}
The number of critical points of $\ell$ over $S^2(\mathbb{C}^n)$ is known as the \emph{ML-degree} of $G$ and we denote it by $\MLdeg(G)$. This is well-defined as long as the matrix $S$ is generic, see \cite{AmendolaEtAl, SethBook, StUh, dms2021}. Geometrically, the ML-degree is the number of intersection points of $(S+L_G^{\perp})\cap L_{G}^{-1}$ with $S$ a generic symmetric matrix in $S^2(\CC^n)$. Sometimes it is more convenient to work projectively, as the number of intersection points remains the same. When the graph is the $n$-th cycle $C_n$, one can compute that
$$\MLdeg(C_n)=1,5,17,49,129,321 \, \text{ for } n=3,4,5,6,7,8.$$
Based on this, Drton, Sturmfels and Sullivant posed the following:
\begin{Conjecture}(\cite[Section 7.4]{drtonSS})\label{conj}
$$\MLdeg(C_n)=(n-3)\cdot 2^{n-2}+1.$$
\end{Conjecture}

Of special interest is the case when the ML-degree is one (as it holds for the 3-cycle $C_3$), as in this case the MLE can be written as a rational function in the entries of $S$.
\begin{Definition}
    A graph $G$ is \emph{chordal} (or decomposable) if every induced cycle of length at least 4 contains a chord.
\end{Definition}

Chordal graphs are of great interest as they admit a decomposition as a clique sum of complete graphs, and as such one can show that their ML-degree is one \cite[Lemma 1]{StUh}. The converse implication sketched in \cite[Theorem 3]{StUh} is more subtle, and in this paper we provide a detailed argument to show why non-chordal graphs have ML-degree larger than one (verified partially in \cite[Lemma 2.4.7]{uhler2011geometry} for $n$-cycles with $n\leq 12$). Our first Main Theorem \ref{thm:chorML1} is the following.
\begin{TheoremMain}
A graph $G$ is chordal if and only if $\MLdeg (G) = 1$.
\end{TheoremMain}
It is worth noting that an analogous characterization is known to hold for \emph{discrete} graphical models, see \cite[Theorem 4.4]{GeigerMeekSturmfels}.

Moreover, we prove a lower bound on the ML-degree of the $n$-cycle (see Theorem \ref{ML_odd_bound}), which matches the value from Conjecture~\ref{conj}. We obtain our second Main Theorem \ref{ML_odd_bound}:\\
\textbf{Theorem.} $$\MLdeg(C_n)\geq (n-3)\cdot 2^{n-2}+1.$$

The paper is structured as follows. We prove that the ML-degree behaves monotonically with respect to subgraphs in Section \ref{sec:mld1} (Corollary \ref{cor:subgr}). The other ingredient is to lower bound the ML-degree of the $n$-cycle; we address this in Section \ref{sec:lowerbound}. This completes the characterization of Gaussian graphical models with ML-degree one as precisely the ones corresponding to chordal graphs.

Our results should provide insight towards proving \cite[Conjecture 2.16]{cox2024homaloidal}, which concerns the slightly more general H\"usler-Reiss graphical models and states that the ML-degree one models are still precisely the ones corresponding to chordal graphs.

\section*{Acknowledgements}
Mateusz Micha{\l}ek thanks the Institute for Advanced Study for  a great working environment and support through the Charles Simonyi Endowment. Part of this work happened at the Combinatorial Coworkspace 2024 in Haus Bergkranz. 


\section{ML-degree one and chordal graphs}\label{sec:mld1}
\subsection{ML-degree of induced subgraphs}
In this subsection, we prove that the ML-degree of a graph is greater or equal to the ML-degree of any induced subgraph.
We will denote by $\mathcal{P}_2(n)$ the set of pairs $\{i,j\}$ with $i\le j \in [n]$, and by $\mathcal{N}(v)$ the set of neighbors of the vertex $v$, excluding the vertex $v$ itself.
\begin{Lemma}\label{Ml-subgr}
Let $G$ be a graph with a vertex $v$. Let $H$ be an induced subgraph with the vertex set $V(G)\setminus\{v\}$. Then $$\MLdeg(G)\ge \MLdeg(H).$$
\end{Lemma}
\begin{proof}

For an $n\times n$ matrix $M$, let us denote by $\overline{M}$ the extended matrix: $$\overline{M}=\begin{pmatrix}
M&0\\
0&1
\end{pmatrix}.$$

Let $n:=|V(H)|=|V(G)|-1$. Consider a generic symmetric matrix $S\in S^2(\CC^n)$. From the definition of the ML-degree, we know that the set $(S+L_H^\perp)\cap L_H^{-1}$ consists of $\MLdeg(H)$ many smooth points.

Consider any point $A\in ({S}+L_H^\perp)\cap L_H^{-1}$. We will show that the extended point $\overline{A}$ lies in $(\overline{S}+L_G^\perp)\cap L_G^{-1}$ and that it is smooth, and therefore an isolated point in the intersection.
This will imply that the set $(\overline{S}+L_G^\perp)\cap L_G^{-1}$ contains at least $\MLdeg(H)$ isolated points from which we can conclude that $\MLdeg(G)\ge \MLdeg(H)$. As $\overline{S}$ is not generic, we cannot conclude that the cardinality of the intersection  $(\overline{S}+L_G^\perp)\cap L_G^{-1}$ equals $\MLdeg(G)$. However, having sufficiently many isolated points in the intersection proves the desired inequality.

Since $S$ is generic, we may assume that the matrix $A$ is regular. Therefore, it has an inverse and $A^{-1}\in L_H$. Clearly, $\overline{A^{-1}}\in L_G$. Moreover, $\overline{A^{-1}}^{-1}=\overline{A}$. This means that $\overline{A}$ is regular and belongs to $L_G^{-1}$.
Similarly, we have $A-S\in L_H^\perp$ which implies $\overline{A}-\overline{S}\in L_G^\perp$. It remains to show that $\overline{A}$ is a smooth point in the intersection.

 Let $D=\{\{i,i\}:i=1,\dots, |V(G)|\}$ be a collection of multisets, each one identified with a loop at the $i$-th vertex of $G$. Since every non-zero, non-diagonal entry in a matrix from $L_G$ corresponds to an edge of $G$ we can naturally parametrize the space $L_G$ by $\{x_e\}_{e\in E(G)\cup D}$. Let us denote by $M_G$, the general parametrizing matrix of $L_G$, i.e $M_{(i,j)}=0$, if $\{i,j\}\not\in E(G)\cup D$, and $M_{(i,j)}=x_e$, if $e=\{i,j\}\in E(G), i\neq j$ and $M_{(i,i)}=2x_{\{i,i\}}$. \footnote{The coefficient two for parametrizing diagonal elements is there only for technical reasons since it makes the computations simpler to write.}

 For any matrix $M$ and two sets $I,I'\subset\{1,\dots,n\}$ we denote by $M(I,I')$ the matrix obtained by deleting the rows indexed by $I$ and columns indexed by $I'$.

 For a pair $e=\{i,j\}$ let us denote by $\gamma_e^G$ the polynomial $$\gamma_e^G=(-1)^{i+j}\det(M_G(\{i\},\{j\})),$$
that is, $\gamma_e^G$ is the signed $(n-1)\times (n-1)$ minor of $M_G$.  Note that we do not require $e$ to be an edge of $G$ in this definition.

 Similarly, for two pairs $e_1=\{i_1,j_1\},e_2=\{i_2,j_2\}$, we denote by
 \begin{small} $$\gamma^G_{e_1,e_2}=(-1)^{i_1+i_2+j_1+j_2}\det(M_G(\{i_1,i_2\},\{j_1,j_2\}))+(-1)^{i_1+i_2+j_1+j_2}\det(M_G(\{i_1,j_2\},\{j_1,i_2\})).$$ \end{small}
In the case that the two elements of a set (e.g.~$i_1=i_2$) are equal, we consider the determinant to be zero.

 To simplify the computation we pass to the projective space $\PP(S^2(\CC^n))$. It is sufficient to show that $\overline{A}$ is smooth after projectivization. The rational map $\varphi : L_G\rightarrow L_G^{-1}$ which sends a matrix to its inverse, is given in coordinates by the polynomials $\gamma^G_e$ for $e\in \mathcal P_2(n)$. We recall that, as $\overline{A}$ is invertible, the map $\varphi$ is an isomorphism of a neighbourhood of $\overline{A}^{-1}$ and a neighbourhood of $\overline{A}$ intersected with $L_G^{-1}$.
 Thus, to show that point $\overline{A}$ is smooth in the intersection we need to prove transversality of two tangent spaces. The first one is the tangent space that is the row span of the Jacobian matrix $J_\varphi$ of the map $\varphi$ evaluated at $\overline{A}^{-1}$.  The second one is simply $L_G^{\perp}$. We can split $J_\varphi$ in two submatrices $ J_\varphi^{E(G)}$ and $J_\varphi^{NE(G)}$, such that the matrix $ J_\varphi^{E(G)}$ consists only of rows that are partial derivatives of $\gamma^G_e$ for $e\in E(G)\cup D$ and $J_\varphi^{NE(G)}$ consists of partial derivatives of $\gamma_e^G$ for $e\not\in E(G)\cup D$. The latter corresponds to coordinates on $L_G^{\perp}$.

Thus, the point $\overline{A}$ is smooth in the intersection 
%
%
if and only if the matrix $J_\varphi^{E(G)}$ evaluated at $\overline{A}^{-1}$ is regular. From now on, we denote $J_\varphi^{E(G)}$  simply by $J_G$. It is a square matrix of size $|E(G)\cup D|\times |E(G)\cup D|$ and entries $\frac{\partial \gamma_e^G}{\partial x_{f}}$ for $e,f\in E(G)\cup D$. Note that $$\frac{\partial \gamma_e^G}{\partial x_{f}}=\gamma_{e,f}^G.$$

We will compute these partial derivatives by considering the following cases:
\begin{itemize}
    \item Suppose $e,f$ are also edges of $H$ or loops not on $v$, i.e.~they do not contain a vertex $v$. Then $\gamma_{e,f}^G (\overline{A}^{-1})=\gamma_{e,f}^H (A^{-1})$. This follows from the fact $\det(M)=\det(\overline{M})$ for any matrix $M$.
    \item Suppose that $e$ is an edge of $H$, $f$ is not an edge and, moreover, $f$ is not a loop,  i.e $f=\{i,n\}$. Then we have that $\gamma_{e,f}^G=0$. To see this, note that after deleting the $n$-th row from $\overline{A}^{-1}$, the last column will be 0, and therefore any determinant would be 0.
    \item Suppose that both of $e,f$ are not edges of $H$ and that they are not loops, i.e. $e=\{i,n\}$, $f=\{j,n\}$. Then $\gamma_{e,f}^G=\gamma_{\{i,j\}}^H$.
    \item Suppose that $e$ is an edge  of $H$ and $f$ is a loop on $v$. Then $\gamma_{e,f}^G=\gamma_e^H$.
    \item Suppose that both $e,f$ contains $v$ and $f$ is a loop. Then $\gamma_{e,f}^G=0$ by definition.
\end{itemize}

Thus, the matrix $J_G(\overline{A}^{-1})$ looks as follows:

$$J_G(\overline{A}^{-1})=\begin{pmatrix}
    J_H&0&(\gamma_e^H)_{e\in E(H)}\\
    0&K&0\\
    (\gamma_e^H)_{e\in E(H)}&0&0
\end{pmatrix}(A^{-1})$$
where $K$ is the matrix whose entries are $\gamma_e^H$ for $e\in P_2(\mathcal N(v))$.

Therefore, it is sufficient to show the regularity of the two matrices
\begin{equation}\label{eq:twomatrices}
    K(A^{-1}) \text{ and } \begin{pmatrix}
    J_H&(\gamma_e^H)_{e\in E(H)}\\
    (\gamma_e^H)_{e\in E(H)}&0&
\end{pmatrix}(A^{-1}).
\end{equation}

We begin by showing that the first matrix $K$ evaluated at $A^{-1}$ is regular. Note that $K(A^{-1})$ is equal to a submatrix of $A$ where one takes the rows and columns indexed by $\mathcal N(v)$ (up to scaling by $\det(A)$). However, the set of all matrices in $L_H^{-1}$ for which this submatrix is singular forms a subvariety of $L_H^{-1}$. It is not equal to the whole $L_H^{-1}$ since $\Id\in L_H^{-1}$ and any square submatrix along the diagonal of the identity is regular. Since $S$ is generic, we may assume that the intersection $(S+L_H^\perp)\cap L_H^{-1}$ does not contain any point from this proper closed subvariety. Thus, for generic $S$, $K(A^{-1})$ is regular.

Now we look at the other matrix. We perform the following row operations on it: we multiply the last row by $(n-1)$ and, for any $e\in E(H)$, we subtract $(A^{-1})_e$-multiple of the row indexed by $e$ from the last row.

In this way, the entry in the last row and $f=\{i_0,j_0\}$-th column after these operations~is
\begin{small}
\begin{align*}
    \left((n-1)\gamma_f^H \vphantom{\sum_{e\in E(H)}} \right. &- \left.\sum_{e\in E(H)} (A^{-1})_e \gamma^H_{e,f}\right)(A^{-1})= (n-1)\gamma_f^H(A^{-1}) - (-1)^{i_0+j_0} \cdot\\  & \sum_{\{i,j\}\in E(H)}(A^{-1})_{i,j}\left( (-1)^{i+j}\det(A^{-1}(\{i_0,i\},\{j_0,j\}))+(-1)^{i+j}\det(A^{-1}(\{i_0,j\},\{j_0,i\}))\right) \\
    &=(n-1)\gamma_f^H(A^{-1})-(-1)^{i_0+j_0}\sum_{\substack{1\le i\le n\\i\neq i_0}}\sum_{\substack{1\le j\le n\\j\neq j_0}}(A^{-1})_{i,j}\left( (-1)^{i+j}\det(A^{-1}(\{i_0,i\},\{j_0,j\}))\right) \\
    &=(n-1)\gamma_f^H(A^{-1})-(-1)^{i_0+j_0}\sum_{\substack{1\le i\le n\\i\neq i_0}}^n\det(A^{-1}(\{i_0\},\{j_0\})) \\
    &=(n-1)\gamma_f^H(A^{-1})-(-1)^{i_0+j_0}(n-1)\det(A^{-1}(\{i_0\},\{j_0\}))=0.
\end{align*}
\end{small}


On the other hand, the value of the bottom right corner is

\begin{align*}
    0-\sum_{e\in E(H)}(A^{-1})_e\gamma^H_{e}(A^{-1})&=
    -\sum_{i=1}^n\sum_{j=1}^n (-1)^{i+j}\det(A^{-1}(\{i\},\{j\}) \\
    &=-\sum_{i=1}^n \det(A^{-1})=n\det(A^{-1})\neq 0.
\end{align*}


Thus, after the row operations the second matrix in \eqref{eq:twomatrices} has the form
$$\begin{pmatrix}
    J_H(A^{-1})&(\gamma_e^H)_{e\in E(H)}(A^{-1})\\
    0&n\det(A^{-1})&
\end{pmatrix}.$$
This allows us to conclude that the matrix is regular because the matrix $J_H(A^{-1})$ is itself regular thanks to $A$ being a smooth point in $L_H^{-1}\cap (S+L_H^{\perp})$.
We can now conclude the regularity of the matrix $J_G(\overline{A}^{-1})$, which completes the proof.
\end{proof}

As a direct consequence, we obtain the following important corollary.
\begin{Corollary}\label{cor:subgr}
Let $G$ be a graph and $H$ an induced subgraph. Then $$\MLdeg(G)\ge \MLdeg(H).$$
\end{Corollary}
\begin{proof}
Follows from Lemma~\ref{Ml-subgr} by induction on $|V(G)|-|V(H)|$.
\end{proof}

\subsection{Lower bounds on number of points in the fiber}
For a rational map $f:X\dashrightarrow Y$ and $y\in Y$ we call the \emph{fiber over $y$} the locus of $x\in X$, such that $f$ is well defined on $x$ and $f(x)=y$. Assuming $X$ is projective, the fiber over $y$ union the base locus is a closed subset of $X$.
\begin{Lemma}\label{lem:plusone}
Let $f:X\dashrightarrow Y$ be a rational, dominant, generically finite map of degree $d$ of projective varieties. Let $y\in Y$. Suppose that the fiber $f^{-1}(y)$ over $y$ strictly contains at least $k$ connected components whose closure does not intersect the base locus. Then $d\geq k+1$.
\end{Lemma}
\begin{proof}
Let $\Gamma\subset X\times Y$ be the closure of the graph of $f$. Apply Stein factorization \cite[Chapter III, Corollary 11.5]{Hartshorne} to the map: $\pi_2:\Gamma\rightarrow Y$, obtaining $g:\Gamma\rightarrow Y'$ and $h:Y'\rightarrow Y$, where $h$ is finite and $g$ has connected fibers. By assumptions, the preimage $h^{-1}(y)$ must contain at least $k$ points, corresponding to the $k$ connected components, and one more point, as the $k$ components are \emph{strictly} contained in the fiber. Thus, $h$ is of degree at least $k+1$. Hence, the preimage of a general point in $Y$ by $h$ contains at least $k+1$ points, which give rise to at least $k+1$ points in the general fiber of $f$.
\end{proof}
\subsection{Characterization of graphical models with ML-degree one}
\begin{Theorem}\label{thm:chorML1}
A graph $G$ is chordal if and only if $\MLdeg (G) = 1$.
\end{Theorem}
\begin{proof}
The implication that chordal graphs have ML-degree equal to one was already proved by Sturmfels and Uhler. Indeed, any chordal graph is a clique sum of complete graphs. The ML-degree of a clique is one and the ML-degree is multiplicative with respect to clique sums \cite[Lemma 4.2]{StUh}.

We now prove the other implication. By definition, every graph that is not chordal contains an induced chordless cycle $H$ of length $n>3$. By Corollary \ref{cor:subgr} it is enough to prove that $\MLdeg(H)>1$. Our proof is by looking at the fiber of the map $\PP(L_H^{-1})\dashrightarrow \PP(S^2\CC^n/L_H^{\perp})$ over the (class of the image of the) identity. In the next section we prove how many isolated points there are in the fiber, however this requires a quite technical argument. Here, for the sake of simplicity, we sketch a qualitative, but not quantitative, simpler argument. By Lemma~\ref{lem:plusone}, it is enough to prove the following two statements:
\begin{enumerate}
\item identity is an isolated point in the fiber,
\item there exists a point in the fiber, that is different from the identity.
\end{enumerate}
To prove the first point, parameterize the neighborhood of the identity in the fiber by symmetric matrices $(m_{i,j})_{1\leq i\leq j\leq n}$, where $m_{i,i}=1$, $m_{i,i+1}=0$, $m_{1,n}=0$ and the other parameters are free variables. As we pick a small Euclidean neighborhood, we may assume $|m_{i,j}|<1$ for $i\neq j$. Our aim is to prove that $m_{i,j}=0$ when $i\neq j$, if $(m_{i,j})\in \PP(L_H^{-1})$.
By \cite[Theorem 1.1]{conner2023sullivant}, we know the polynomial equations that $m_{i,j}$ must satisfy. In particular, for any $1\leq i\leq n$ and $j\neq i+1$ the minor of the $3\times 3$ submatrix  of $(m_{i,j})$ with columns $i,i+1,i+2$ and rows $i,i+2,j$, where we use cyclic notation modulo $n$, must vanish. Hence, we have, $m_{i+1,j}(1-m_{i,i+2}^2)=0$, which under our assumptions implies $m_{i+1,j}=0$. This finishes the proof of the first point.

For the second point, if $n$ is even, we may take $m_{i,j}=(i-j)$ mod $2$, as a second point in the fiber.
If $n$ is odd, the argument is more involved and here we present just a sketch of the proof, referring for more detailed results to Section~\ref{sec:lowerbound}. We consider a family of matrices $M(x)=(m_{i,j})\in L_H$ parameterized by $x$, where $m_{i,i}=1$, $m_{i,i+1}=x$, $m_{1,n}=x$ and $m_{i,j}=0$ otherwise. Our aim is to find $x\neq 0$ such that $\det M(x)\neq 0$, $\det M(x)(\{1\},\{1\})\neq 0$ and $\det M(x)(\{1\},\{2\})=0$, where $M(x)(\{i\},\{j\})$ is $M(x)$ with the $i$-th row and the $j$-th column removed. Indeed, in such a case $M(x)^{-1}$ is a point in the fiber that is not the identity matrix. As $n$ is odd, we have $\det M(x)(\{1\},\{2\})=-\det M(x)(\{1\},\{n\})$. Thus, $\det M(x)= M(x)(\{1\},\{1\})$ and it is enough to prove that there exists a nonzero root of $\det M(x)(\{1\},\{2\})$ that is not a root of $\det M(x)(\{1\},\{1\})$. For this, we note that $ M(x)(\{1\},\{2\})$ has first row divisible by $x$. After dividing by $x$ and setting $x=0$ the determinant is nonzero. Thus $x=0$ is a simple root of  $\det M(x)(\{1\},\{2\})$ and there must exist other roots. In Lemma~\ref{lemmas}(b) we prove that $\det M(x)(\{1\},\{2\})$ is a product of two polynomials. One of them divides $\det M(x)(\{1\},\{1\})$ and the other is coprime with it. Thus, by taking $x$ to be a root of the second polynomial we conclude the proof.
\end{proof}

In the next section we provide an explicit lower bound for the ML-degree of the general $n$-cycle $C_n$. 

\section{Lower bound on ML-degree of cyclic models}\label{sec:lowerbound}
In this section, we prove a lower bound for the ML-degree of a cycle, which partially confirms a conjecture of Drton, Sturmfels, and Sullivant, \cite[Section 7.4]{drtonSS}. Our main task is to achieve a good understanding of the set $L_{C_n}^{-1}\cap (\Id+L_{C_n}^\perp)$, as the cardinality of this set lower bounds the given ML-degree. In the first step, we introduce groups that naturally act on our algebraic sets, which will be used later for symmetry arguments.

\subsection{General facts}

Consider the set $\mathcal D_n^{\pm}$ of all $n\times n$ diagonal matrices whose diagonal entries are equal to $\pm 1$. This is a group under multiplication and for any $D\in\mathcal D_n^\pm$, we have $D=D^{-1}$. This group acts on $S^2(\CC^n)$ by conjugation. Clearly, the spaces $L_{C_n}$ and $L_{C_n}^{\perp}$ are invariant with respect to this group action. In addition, the set $L_{C_n}^{-1}$ is also invariant, since $(DAD)^{-1}=DA^{-1}D$ for all $D\in\mathcal D^{\pm}_n$ and $A\in S^2(\CC^n)$.

Thus, the intersection $L_{C_n}^{-1}\cap (\Id+L_{C_n}^\perp)$ is also invariant. This allows us to characterize the points in this intersection up to the action of the group $\mathcal D_n^\pm$.
Let us denote by $D_n^i$ the element of $\mathcal D^{\pm}_n$ with only one entry equal to $-1$, which is precisely in the $i$-th row.

Let us consider another group action, namely the action by cyclic shift. For this, we consider the following matrices:

$$N_n^+:=\begin{pmatrix}
0&1&0&\dots&0&0&0\\
0&0&1&\dots&0&0&0\\
0&0&0&\dots&0&0&0\\
\vdots&\vdots&\vdots&\ddots&\vdots&\vdots&\vdots\\
0&0&0&\dots&0&1&0\\
0&0&0&\dots&0&0&1\\
1&0&0&\dots&0&0&0
\end{pmatrix}, \
N_n^-:=\begin{pmatrix}
0&1&0&\dots&0&0&0\\
0&0&1&\dots&0&0&0\\
0&0&0&\dots&0&0&0\\
\vdots&\vdots&\vdots&\ddots&\vdots&\vdots&\vdots\\
0&0&0&\dots&0&1&0\\
0&0&0&\dots&0&0&1\\
-1&0&0&\dots&0&0&0
\end{pmatrix}.
$$

Note that $(N_n^+)^{-1}=(N_n^+)^T$ and $(N_n^-)^{-1}=(N_n^-)^T$. Thus, (the groups generated by) these matrices act on $S^2(\CC^n)$ by conjugation and $L_{C_n}$ is invariant with respect to that action.
We say that a matrix $A$ is \textit{$N_n^+$-invariant} if $N_n^+A(N_n^+)^{-1}=A$. Analogously, $A$ is \textit{$N_n^-$-invariant} if $N_n^-A(N_n^-)^{-1}=A$.

From now on, we will denote the entries of the matrix $A$ by $a_{i,j}$ where $0\le i,j\le n-1$, i.e. we index rows and columns from 0. In general we will consider indices modulo $n$, so that $a_{n+i,n+j}=a_{i+j}$.


\begin{Lemma}\label{cyclic}
Let $A\in L_{C_n}^{-1}\cap (\Id+L_{C_n}^\perp)$. Then there exists a matrix $D\in\mathcal D_n^\pm$ such that the matrix $DAD$ is either $N_n^+$-invariant or $N_n^-$-invariant.
\end{Lemma}

\begin{proof}
Consider a matrix $A\in L_{C_n}^{-1}\cap (\Id+L_{C_n}^\perp)$. Clearly we have $a_{i,i}=1$ and $a_{i,i+1}=0$ for all $0\le i\le n-1$, where $a_{n-1,n}=a_{0,n-1}$. Since $A\in L_{C_n}^{-1}$, it satisfies all equations from $I_{C_n}$. We will use some of them to prove the result:

$$\delta(0,1,2)(0,2,3):=\det\begin{pmatrix}
1&a_{0,2}&a_{0,3}\\
0&0&a_{1,3}\\
a_{0,2}&1&0
\end{pmatrix}=a_{1,3}(a_{0,2}^2-1).$$

Thus, we have $a_{1,3}=0$ or $a_{0,2}=\pm1$. By cyclic shift we get $a_{i,i+2}=0$ or $a_{i-1,i+1}=\pm1$ for all $0\le i\le n-1$, where indices are taken modulo $n$.

Now there are two options that occur: either $a_{i,i+2}=0$ for all $0\le i\le n-1$ or there exists an index $i_0$ such that $a_{i_0,i_0+2}\neq 0$. However, in the second case, from the above equation, we get $a_{i_0-1,i_0+1}=\pm1\neq 0$ and, by repeating this procedure, we obtain that $a_{i,i+2}=\pm 1$ for all $0\le i\le n-1$, where we again consider the indices modulo $n$.

We first look at the case $a_{i,i+2}=0$ for all $0\le i\le n-1$. We will show that $a_{i,i+d}=0$ for all $0\le i,d\le n-1$. For $d=1,2$ it holds. Consider the following $3\times 3$ minor:

$$0=\delta(i-1,i,i+1)(i-1,i+1,i+d):=$$
$$=\det
\begin{pmatrix}
1&a_{i-1,i+1}&a_{i-1,i+d}\\
0&0&a_{i,i+d}\\
a_{i+1,i-1}&1&a_{i+1,i+d}
\end{pmatrix}=\det
\begin{pmatrix}
1&0&a_{i-1,i+d}\\
0&0&a_{i,i+d}\\
0&1&a_{i+1,i+d}
\end{pmatrix}=a_{i,i+d}.$$

Therefore, in this case, we must have $A=\Id$ which is both $N_n^+$ and $N_n^-$-invariant.

Now let us consider the case $a_{i,i+2}=\pm 1$ for all $0\le i\le n-1$. Note that, by acting with $D_n^i$, we change the signs of $a_{i-2,i}$ and $a_{i,i+2}$. Thus, we can go through the entries $a_{0,2},a_{1,3},\dots a_{n-3,n-1}$ and by acting, if needed, with $D_n^2,D_n^3,\dots,D_n^{n-1}$, we make them all equal to $-1$. Thus, we may assume $a_{0,2}=a_{1,3}=\dots=a_{n-3,n-1}=-1$.

Next, we consider $3\times 3$ minors:

$$0=\delta(1,2,4)(0,1,4)=\det\begin{pmatrix}
0&1&a_{1,4}\\
-1&0&-1\\
a_{0,4}&a_{1,4}&1
\end{pmatrix}=-a_{1,4}^2-a_{0,4}+1.$$

Thus, $a_{1,4}^2=1-a_{0,4}$. Analogously, from $\delta(0,2,3)(0,3,4)$ we get $a_{0,3}^2=1-a_{0,4}$. Thus, $a_{0,3}^2=a_{1,4}^2$. By induction and a cyclic shift of the last equation, we get $a_{i,i+3}^2=a_{0,3}^2$ for all $0\le i\le n-4$. Moreover, by a cyclic shift of the equation $a_{0,3}^2=1-a_{0,4}$, we obtain $a_{i,i+4}=1-a_{i+1,i+4}^2=1-a_{0,3}^2$ for all $0\le i\le n-5$.

Our goal will be to prove that $a_{i,i+3}=a_{0,3}$ for all $0\le i\le n-4$. For this, we consider another $3\times 3$ minor:

$$0=\delta(0,1,2)(0,3,4)=\det\begin{pmatrix}
1&a_{0,3}&a_{0,4}\\
0&-1&a_{1,4}\\
-1&0&-1
\end{pmatrix}=$$
$$=1-a_{0,3}a_{1,4}-a_{0,4}=1-a_{0,3}a_{1,4}-1+a_{0,3}^2=a_{0,3}(-a_{1,4}+a_{0,3}).$$

If $a_{0,3}=0$, then we have $a_{i,i+3}=a_{0,3}=0$ for all $i\le n-4$. Thus, let us assume $a_{0,3}\neq 0$. Then $a_{1,4}=a_{0,3}$. Analogously (by cyclic shift), we obtain that $a_{i+1,i+4}=a_{i,i+3}$ for all $0\le i \le n-5$, therefore, by induction, we have that $a_{i,i+3}=a_{0,3}$ for all $0\le i\le n-4$.

Next, we prove that the matrix $A$ has equal entries on all diagonals, i.e $a_{i,i+d}=a_{0,d}$ for all $i,d\ge 0, i+d \le n-1$. We will proceed by induction on $d$. For $d=1,2,3$ we have already shown it. Next, consider $d\ge 4$ and we assume that the statement is true for all $d'<d$. We compute the following minor:
$$0=\delta(i,i+1,i+2)(i,i+3,i+d)=
\det\begin{pmatrix}
1&a_{0,3}&a_{i,i+d}\\
0&-1&a_{0,d-1}\\
-1&0&a_{0,d-2}
\end{pmatrix}=-a_{0,d-2}-a_{0,3}a_{0,d-1}-a_{i,i+d}.
$$

Thus $a_{i,i+d}=-a_{0,d-2}-a_{0,3}a_{0,d-1}$. Since this is true for all $i$ and the right-hand side does not depend on $i$, we have that $a_{i,i+d}=a_{0,d}$ for all $0\le i\le n-d-1$.

Consider the entry $a_{0,n-2}=\pm 1$. First, assume that $a_{0,n-2}=-1$ and consider the matrix $A'=(N_n^+)A(N_n^+)^{-1}\in L_{C_n}^{-1}\cap (\Id+L_{C_n}^\perp)$. We have $a'_{0,2}=\dots=a'_{n-1,n-3}$, thus by the same arguments as in the case of matrix $A$, we can prove $a'_{i,i+d}=a'_{0,d}$ for all $0\le i\le n-d-1$. This means that $a_{0,n-d}=a_{0,d}$ for all $d\le n/2$ and the matrix $A$ is $(N_n^+)$-invariant.

The case $a_{0,n-2}=1$ is analogous, the only difference is that we consider the matrix $A'=(N_n^-)A(N_n^-)^{-1}$ instead and we get that the matrix $A$ is $(N_n^-)$-invariant.
\end{proof}

\begin{Remark}\label{noneed}
Note that, for odd $n$, there exists a matrix $D\in \mathcal D^\pm_{n}$ such that $DN_n^+D=-N_n^-$. The matrix $D$ has $-1$ on all even positions. That implies that, for $N_n^-$-invariant matrix $A$, we have
$$DAD=DN_n^-A(N_n^-)^{-1}D=-(N_n^+D)A(-D(N_n^+)^{-1})=N_n^+(DAD)(N_n^+)^{-1}.$$

Hence, the matrix $DAD$ is $N_n^+$-invariant. This means that, for odd $n$, a stronger version of Lemma~\ref{cyclic} holds, as we do not need the $N_n^-$-invariant assumption. However, that is not the case for even $n$.
\end{Remark}

Let us denote by $M_n(x),M_n^+(x),M_n^-(x)$ the following matrices:

$$M_n(x):=\begin{pmatrix}
1&x&0&0&\dots&0&0\\
x&1&x&0&\dots&0&0\\
0&x&1&x&\dots&0&0\\
0&0&x&1&\dots&0&0\\
\vdots&\vdots&\vdots&\vdots&\ddots&\vdots&\vdots\\
0&0&0&0&\dots&1&x\\
0&0&0&0&\dots&x&1
\end{pmatrix},$$

$$M_n^+(x):=\begin{pmatrix}
1&x&0&0&\dots&0&x\\
x&1&x&0&\dots&0&0\\
0&x&1&x&\dots&0&0\\
0&0&x&1&\dots&0&0\\
\vdots&\vdots&\vdots&\vdots&\ddots&\vdots&\vdots\\
0&0&0&0&\dots&1&x\\
x&0&0&0&\dots&x&1
\end{pmatrix},\ M_n^-(x):=\begin{pmatrix}
1&x&0&0&\dots&0&-x\\
x&1&x&0&\dots&0&0\\
0&x&1&x&\dots&0&0\\
0&0&x&1&\dots&0&0\\
\vdots&\vdots&\vdots&\vdots&\ddots&\vdots&\vdots\\
0&0&0&0&\dots&1&x\\
-x&0&0&0&\dots&x&1
\end{pmatrix}.$$

Moreover, let $P_n(x):=\det(M_n(x))$. We will prove several formulas concerning $P_n(x)$ that will be useful later.

\begin{Lemma}\label{det-recurrence}
It holds that $P_{m+n}(x)=P_m(x)P_n(x)-x^2P_{m-1}(x)P_{n-1}(x)$.
\end{Lemma}
\begin{proof}
We write the matrix $M_{m+n}(x)$ as follows:

$$M_{m+n}(x)=\left(\begin{array}{c|c}
M_m(x)&\begin{matrix}
    0&\dots&0 \\
    &\ddots&\vdots\\
    \color{red}{x}& & 0
\end{matrix}\\
\hline
\begin{matrix}
    0& &\color{red}{x} \\
    \vdots&\ddots&\\
    0&\dots & 0
\end{matrix}&M_n(x)
\end{array}\right)$$

If we expand the determinant, there will be two types of terms. First, there are those that do not contain a red $x$. They will give us $\det(M_m(x))\cdot \det(M_n(x))$. If a term contains a red $x$, then it must also contain the other one, and those terms give us $-x^2\det(M_{m-1}(x))\cdot \det(M_{n-1}(x))$. Hence, $P_{m+n}(x)=P_m(x)P_n(x)-x^2P_{m-1}(x)P_{n-1}(x)$.
\end{proof}

\begin{Corollary}\label{recurrence}
    The following two identities hold for all $n>1$:
    \begin{itemize}
        \item[(a)]  $P_{2n}(x)=(P_n(x)-xP_{n-1}(x))(P_n(x)+xP_{n-1}(x))$
        \item[(b)] $P_{2n-1}(x)=P_{n-1}(x)(P_n(x)-x^2P_{n-2}(x))$.
    \end{itemize}
\end{Corollary}
\begin{proof}
    For (a), by Lemma \ref{det-recurrence}, $$P_{2n}(x)=P_n(x)^2-x^2P_{n-1}(x)^2=(P_n(x)-xP_{n-1}(x))(P_n(x)+xP_{n-1}(x)).$$
    For (b), plug in $m=n-1$ in Lemma \ref{det-recurrence}.
\end{proof}
\begin{Lemma}\label{lemmas}
The following statements hold for all $n>1$:
\begin{itemize}
    \item[(a)] $x^{2n-2}=P_{n-1}(x)^2-P_n(x)P_{n-2}(x)$ 
    \item[(b)] $P_{2n-1}(x)-x^{2n-1}=(P_n(x)-xP_{n-1}(x))(P_{n-1}(x)+xP_{n-2}(x))$ 
    \item[(c)] $P_{2n}(x)+x^{2n}=P_n(x)(P_{n}(x)-x^2P_{n-2}(x))$
    \item[(d)] $P_{2n}(x)-x^{2n}=P_{n-1}(x)(P_{n+1}(x)-x^2P_{n-1}(x))$
    \item[(e)] Polynomials $P_n(x)$, $P_{n-1}(x)$ and $P_{n-2}(x)$ are mutually coprime for all $n>1$. 
\end{itemize}
\end{Lemma}
\begin{proof}
(a) We proceed by induction on $n$. For $n=2$ it is trivial since $P_1(x)=P_0(x)=1$ and $P_2(x)=1-x^2$.
Let us assume that the statement holds for $n-1$. By Lemma \ref{det-recurrence},
$$P_{n-1}^2(x)-x^2P_{n-2}(x)^2=P_{2n-2}(x)=P_n(x)P_{n-2}(x)-x^2P_{n-1}(x)P_{n-3}(x).$$
By comparing the left and right-hand sides and using the induction hypothesis, we obtain:
$$P_{n-1}(x)^2-P_n(x)P_{n-2}(x)=x^2(P_{n-2}(x)^2-P_{n-1}(x)P_{n-3}(x))=x^2\cdot x^{2n-4}=x^{2n-2}.$$
(b) By Lemma~\ref{det-recurrence}, we have
$$P_{2n-1}(x)=P_n(x)P_{n-1}(x)-x^2P_{n-1}(x)P_{n-2}(x).$$
Thus it remains to prove that $$x^{2n-1}=x((P_{n-1}(x))^2-P_n(x)P_{n-2}(x)),$$
which holds by part (a). \\
(c) By Corollary~\ref{recurrence}(a) and part (a), we have that
$$P_{2n}(x)+x^{2n}=P_n(x)^2-x^2P_{n-1}(x)^2+x^2P_{n-1}(x)^2-x^2P_n(x)P_{n-2}(x),$$
which shows the first equality. The proof of (d) is analogous. \\
(e) Note that none of the polynomials $P_n(x)$ is divisible by $x$. Suppose that there is a non-constant polynomial $Q(x)$ such that $Q(x)$ divides two of polynomials $P_n(x)$, $P_{n-1}(x)$ and $P_{n-2}(x)$. By using Lemma~\ref{det-recurrence} for $m=1$, we obtain that:
 $$P_{n}(x)=P_{n-1}(x)-x^2P_{n-2}(x).$$
Thus, $Q(x)$ divides all three polynomials. Using this recurrence, one can easily see by induction that $Q(x)\mid P_m(x)$ for any positive integer $m$. However, $P_1(x)=1$ which is a contradiction.
\end{proof}

\begin{Lemma}\label{regular}
    Let $A\neq \Id$ be a regular $N_n^+$-invariant (resp. $N_n^-$-invariant) matrix from $L_{C_{n}}^{-1}$. Then  $A^{-1}\in \Id+L_{C_{n}}^\perp$ if and only if $A=M_n^+(x)$ and $x^{n-2}+(-1)^n P_{n-2}(x)=0$ (resp. $A=M_n^-(x)$ and $x^{n-2}-(-1)^n P_{n-2}(x)=0)$.
\end{Lemma}
\begin{proof}
Since the inverse matrix is given by $(n-1)\times(n-1)$ minors, we have $A^{-1}\in \Id+L_{C_{n}}^\perp$ if and only if $\det(A(\{0\},\{n-1\}))=0$ and $\det(A(\{0\},\{0\}))=\det(A)$. Note that it is sufficient to check only these two entries, since the matrix $A^{-1}$ is also $N_n^+$-invariant or $N_n^-$-invariant. By Laplace expansion, we get
\begin{equation*}\det(A)=a_{0,0}\det(A(\{0\},\{0\}))-a_{0,1}\det(A(\{0\},\{1\}))+(-1)^{n-1}a_{0,n-1}\det(A(\{0\},\{n-1\})).
\end{equation*}

 Thus, assuming that $\det(A(\{0\},\{n-1\}))=0$, the condition $\det(A(\{0\},\{0\}))=\det(A)$ holds if and only if $a_{0,0}=1$. Therefore, we may assume $A=M_n^+(x)$, (resp. $A=M_n^-(x)$).

We compute the determinant $\det(A(\{0\},\{n-1\}))$ by expansion by the first column. We do only the case of $N_n^+$-invariant matrix, the other case is analogous.
$$\det(A(\{0\},\{n-1\}))=x\det(A(\{0,1\},\{0,n-1\}))+(-1)^{n}x\det(A(\{0,n-1\},\{0,n-1\}))=$$
$$=x^{n-1}+(-1)^n xP_{n-2}(x).$$
Since $A\neq\Id$ we have $x\neq 0$. Thus,  $A^{-1}\in \Id+L_{C_{n}}^\perp$ if and only if $x^{n-2}+(-1)^n P_{n-2}(x)=0$.

In the case of $N_n^-$-invariant matrix, we obtain by similar computation that the condition is $x^{n-2}-(-1)^n P_{n-2}(x)=0$.
\end{proof}

\begin{Lemma}\label{regular-odd/even}
 The characterization of the $N_n^+$-invariant (or $N_n^-$-invariant) regular matrices $A \neq \Id$ in $L_{C_{n}}$ such that $A^{-1}\in \Id+L_{C_{n}}^\perp$  is the following one:

\begin{itemize}
    \item For $n=2m+1$ and $N_n^+$-invariant matrix: $A=M_n^+(x)$, $P_{m-1}(x)+xP_{m-2}(x)=0.$
    \item For $n=2m$ and $N_n^+$-invariant matrix: $A=M_n^+(x)$, $P_{m-1}(x)-x^2P_{m-3}(x)=0.$
    \item For $n=2m$ and $N_n^-$-invariant matrix: $A=M_n^-(x)$, $P_{m-2}(x)=0.$
\end{itemize}
\end{Lemma}

\begin{proof}
By Lemma \ref{regular}, in all cases we have $A=M_n^+(x)$ (or $A=M_n^-(x)$), and a polynomial condition for $x$. However, we must also check that $A$ is regular, i.e. $\det(A)\neq 0$. We have
$$\det(A)=a_{0,0}\det(A(\{0\},\{0\}))-a_{0,1}\det(A(\{0\},\{1\}))+(-1)^{n-1}a_{0,n-1}\det(A(\{0\},\{n-1\}))=$$
$$=\det(A(\{0\},\{0\}))=P_{n-1}(x).$$

Thus, in all cases, we want $x$ not to be a root of $P_{n-1}(x)$. We consider the cases separately:

\begin{description}
\item[For $n=2m+1$ and $N_n^+$-invariant matrix]
We apply Lemma~\ref{lemmas}(b), to get $$0=x^{n-2}-P_{n-2}(x)=(P_m(x)-xP_{m-1}(x))(P_{m-1}(x)+xP_{m-2}(x)).$$ However, $(P_m(x)-xP_{m-1}(x))\mid P_{n-1}(x)$, which shows that $x$ cannot be a root of this polynomial.
In consequence, we must have $P_{m-1}(x)+xP_{m-2}(x)=0$.

We claim that for every root of this polynomial, the matrix $A=M_n^+(x)$ is regular. To prove that, we will show that polynomials $P_{m-1}(x)+xP_{m-2}(x)$ and $P_{n-1}(x)$ are coprime. Indeed, by Corollary~\ref{recurrence}(a), we have that $P_{m-1}(x)+xP_{m-2}(x)\mid P_{2m-2}(x)$. By Lemma~\ref{lemmas}(e), the polynomials $P_{n-1}$ and $P_{2m-2}$ are coprime, which implies that also polynomials $P_{m-1}(x)+xP_{m-2}(x)$ and $P_{n-1}(x)$ are coprime.
\item[For $n=2m$ and $N_n^+$-invariant matrix]
We apply Lemma~\ref{lemmas}(c) to obtain
$$P_{n-2}(x)+x^{n-2}=P_{m-1}(x)(P_{m-1}(x)-x^2P_{m-3}(x)).$$
However, by Corollary \ref{recurrence}(b) $P_{m-1}(x)\mid P_{2m-1}(x)$, thus we must have $P_{m-1}(x)-x^2P_{m-3}(x)=0$. Since this polynomial is a divisor of $P_{2m-3}(x)$ which is, by Lemma \ref{lemmas}(e), coprime with $P_{2m-1}(x)$, every root of this polynomial will result in a regular matrix.
\item[For $n=2m$ and $N_n^-$-invariant matrix]
Analogously, we apply Lemma~\ref{lemmas}(d) to obtain
$$P_{n-2}(x)-x^{n-2}=P_{m-2}(x)(P_{m}(x)-x^2P_{m-2}(x)).$$
This time, we have that $(P_{m}(x)-x^2P_{m-2}(x))\mid P_{2m-1}(x)$, again by Corollary \ref{recurrence}(b). This implies $P_{m-2}(x)=0$. This polynomial is a divisor of $P_{2m-3}(x)$, thus it is coprime with $P_{2m-1}(x)$. \qedhere
\end{description}
\end{proof}

\subsection{Singular matrices}
In this subsection, we characterize the points in $L_{C_{n}}^{-1}\cap(\Id+L_{C_{n}}^{\perp})$, up to the action of the group $\mathcal{D}^{\pm}_{n}$. For this, we need to understand more about the singular matrices in the intersection. The following result is a consequence of the structure of low-rank matrices in $L_{C_n}$ from \cite[Lemmas 3.1, 3.2]{DMV}.
\begin{Lemma}\label{blockform}
Let $B\in L_{C_n}$ be a singular matrix of rank at most $n-3$. Then up to a cyclic shift of rows and columns, the matrix $B$ is in the following form:

$$B=\begin{pmatrix}
    B_1&\dots&0\\
    \vdots&\ddots&\vdots\\
    0&\dots&B_k
\end{pmatrix}$$
Moreover, the matrices $B_i$ are tridiagonal and are either of rank $|B_i|$ or $|B_i|-1$, where $|B_i|$ denotes the size of a matrix.
\end{Lemma}

\begin{Lemma}\label{rank2}
Let $A\in L_{C_{n}}^{-1}\cap(\Id+L_{C_{n}}^{\perp}) $ be a singular matrix. Then the rank of $A$ is at most two.
\end{Lemma}
\begin{proof}

Let us work in the projective space and consider the closure of the graph of the inverse map, i.e:

$$\Gamma:=\overline{\{([B],[B^{-1}]): B\in L_{C_n},\ \det(B)\neq 0)\}}\subset \PP(L_{C_n})\times \PP(S^2(\CC^n)).$$

\begin{equation*}
\begin{tikzcd}
& \Gamma\subset\PP(L_{C_n})\times \PP(S^2(\CC^n)) \arrow[ld,"\pi_1"'] \arrow[rd,"\pi_2"] &  \\
\PP(L_{C_n}) \arrow[rr,""',dashed] & & \PP(\CC^n)
\end{tikzcd}
\end{equation*}

For any pair of regular matrices $([B],[A])\in \Gamma$, we have the equations $(BA)_{i,j}=0$ for $i\neq j$ and $(BA)_{i,i}-(BA)_{j,j}=0$. Therefore, these equations belong to the ideal $I(\Gamma)$.

Since the image of the projection $\pi_2$ restricted to $\Gamma$ is $\PP(L_{C_n}^{-1})$, for any nonzero $A\in L_{C_n}^{-1}$ there exists a nonzero matrix $B\in L_{C_n}$ such that $([A],[B])\in\Gamma$. Thus, $BA=t\Id$ for some $t$. In the case where matrix $A$ is singular, we must have $BA=0$.

Assume for contradiction that we have a singular matrix $A\in L_{C_{n}}^{-1}\cap(\Id+L_{C_{n}}^{\perp})$ such that the rank of $A$ is at least three. Then there exists a nonzero matrix $B\in L_{C_n}$ of rank at most $n-3$ such that $BA=0$. By Lemma~\ref{blockform}, the matrix $B$ is in block form (up to cyclic shift of rows and columns). Clearly, there exists a block of a positive rank, as $B\neq 0$. Without loss of generality, it is the first block $B_1$ and it has size $k$. If $\rank(B_1)=k$, then the first $k$ rows of $A$ must be equal to 0 which is not possible. Thus, $\rank(B_1)=k-1$ and $k\ge 2$. However, in this case, the matrix formed by the first $k$ rows of $A$ must have rank one. Again, this is not possible, since the left upper $2\times 2$ minor of $A$ is equal to one. That is a contradiction, which finishes the proof.
\end{proof}

\begin{Lemma}\label{singular}
For odd $n$ there are no singular matrices in the set $L_{C_{n}}^{-1}\cap(\Id+L_{C_{n}}^{\perp})$. For even $n$, the set of all singular matrices in $L_{C_{n}}^{-1}\cap(\Id+L_{C_{n}}^{\perp})$ is equal to
$$\{D\mathcal{C}_nD : D\in \mathcal D_n^\pm\},$$
where $\mathcal{C}_n$ is the checkerboard matrix, i.e. $$(\mathcal C_n)_{i,j}=\begin{cases}
 0,&\text{ if } i+j \text{ is odd}  \\
 1,&\text{ if } i+j \text{ is even}.
\end{cases}$$
\end{Lemma}
\begin{proof}
 Consider a singular matrix $A\in L_{C_{n}}^{-1}\cap(\Id+L_{C_{n}}^{\perp})$. By Lemma~\ref{rank2}, the rank of $A$ is at most two. However, the first two rows of $A$ are linearly independent. Therefore, the third row must be a combination of the first two. Since $a_{3,2}=0=a_{1,2}=0$ and $a_{2,2}=1$, the third row must be a multiple of the first row.

By repeating this argument we see that all rows $a_1, a_3, a_5, \dots , a_{2k+1}$ are multiples of each other. For odd $n$ it means that $a_1$ is a multiple of $a_2$, which is a contradiction. Thus, for odd $n$, there are no such matrices.

For even $n$, we have that all odd rows are multiples of $a_1$ and all even rows are multiples of $a_2$. This implies that $a_{i,j}\neq 0$ if and only if $2|i-j$. Moreover, by the vanishing of the minor $\delta(i-1,i,i+2)(i,i+1,i+2)$, we must have $a_{i,i+2}^2=1$. This shows that every odd row is equal to $a_1$ or $-a_1$. By acting with a suitable $D\in \mathcal D_n^\pm$, we get to the situation where all even rows and all odd rows are equal. Since $a_{i,i}=1$, this means that after acting we have $A=\mathcal C_n$.

Note that all such matrices actually belong to $L_{C_n}^{-1}$, since every rank two matrix belongs there. The reason is that the variety $L_{C_n}^{-1}$ is defined by vanishing of some $3\times 3$ minors \cite{conner2023sullivant}. This concludes the proof of the lemma.
\end{proof}
\begin{Lemma}\label{odd-intersection}
We have the following equalities:
$$L_{C_{2n+1}}^{-1}\cap(\Id+L_{C_{2n+1}}^{\perp})=\{\Id\}\cup \{DM_{2n+1}^+(x)^{-1}D : D\in\mathcal D_{2n+1}^\pm;P_{n-1}(x)+xP_{n-2}(x)=0\}.$$
$$L_{C_{2n}}^{-1}\cap(\Id+L_{C_{2n}}^{\perp})=\{\Id\}\cup \{DM_{2n}^+(x)^{-1}D : D\in\mathcal D_{2n}^\pm;P_{n-1}(x)-x^2P_{n-3}(x)=0\}\cup $$
$$\cup \{DM_{2n}^-(x)^{-1}D : D\in\mathcal D_{2n}^\pm;P_{n-2}(x)=0\}\cup \{D\mathcal{C}_{2n}D : D\in\mathcal{D}_{2n+1}^\pm\}.$$
\end{Lemma}
\begin{proof}
It is a direct consequence of Lemma \ref{cyclic}, Lemma \ref{regular-odd/even}, and Lemma \ref{singular}. Note that, by Remark~\ref{noneed}, we do not need to consider the case when the matrix is $N^-_{2n+1}$-invariant.
\end{proof}

\begin{Proposition}\label{ML-odd>1}
Let $n\ge 2$ be an integer.
Then
$$\MLdeg(C_{2n+1})\ge 1+2^{2n}\cdot |\{x\in\CC:P_{n-1}(x)+xP_{n-2}(x)=0\}|.$$
$$\MLdeg(C_{2n})\ge 1+2^{2n-2}\cdot |\{x\in\CC:P_{n-1}(x)-x^2P_{n-3}(x)=0\}|+$$
$$+2^{2n-2}\cdot |\{x\in\CC:P_{n-2}(x)=0\}|+2^{2n-2}.$$
In particular, for $n\ge 4$ we have $\MLdeg(C_{n})>1.$
\end{Proposition}
\begin{proof}
\textbf{For the odd case:} For any $x\neq 0$, the stabilizer of $M_{2n+1}^+(x)$ in the action given by $\mathcal D_{2n+1}^\pm$ consists of only two elements: $\Id$ and $-\Id$. Thus, the orbit of $M_{2n+1}^+(x)$ under the action of $\mathcal D_{2n+1}^\pm$ has $2^{2n}$ elements.

 By Lemma~\ref{odd-intersection}, the set $$L_{C_{2n+1}}^{-1}\cap(\Id+L_{C_{2n+1}}^{\perp})$$ consists of $1+2^{2n}\cdot |\{x\in\CC:P_{n-1}(x)+xP_{n-2}(x)=0\}|$ isolated points. For a general cut, the number of intersection points can only increase \cite[Lemma 37.53.7]{stacks-project}, proving the result.

\textbf{For the even case:} the situation is similar. However, the polynomials $P_{n-1}(x)-x^2P_{n-3}(x)$ and $P_{n-2}(x)$ are even. The orbits of $M_{2n}^+(x)$ and $M_{2n}^+(-x)$ under the action of $\mathcal D^\pm_n$ are the same and consist of $2^{2n-1}$ elements. The same is true for the orbits of $M_{2n}^-(x)$ and $M_{2n}^-(-x)$.

The stabilizer of the checkerboard matrix $\mathcal C_{2n}$ consists of four elements. Except for $\Id$ and $-\Id$, there are also the matrices that have $-1$ in all odd or all even positions. Thus, the orbit of $\mathcal C_{2n}$ consists of $2^{2n-2}$ elements.

Putting this together with Lemma~\ref{odd-intersection}, we see that the set $L_{C_{2n}}^{-1}\cap(\Id+L_{C_{2n}}^{\perp})$ consists of
$$1+2^{2n-2}\cdot |\{x\in\CC:P_{n-1}(x)-x^2P_{n-3}(x)=0\}|+$$
$$+2^{2n-2}\cdot |\{x\in\CC:P_{n-2}(x)=0\}|+2^{2n-2}$$
isolated points. This proves the result also in the even case.
\end{proof}

\subsection{Roots of $P_n(x)$}

Let us define the following $n\times n$ matrix:

$$\widetilde{M}_n(\alpha):=\begin{pmatrix}
2\cos\alpha&1&0&0&\dots&0&0\\
1&2\cos\alpha&1&0&\dots&0&0\\
0&1&2\cos\alpha&1&\dots&0&0\\
0&0&1&2\cos\alpha&\dots&0&0\\
\vdots&\vdots&\vdots&\vdots&\ddots&\vdots&\vdots\\
0&0&0&0&\dots&2\cos\alpha&1\\
0&0&0&0&\dots&1&2\cos\alpha
\end{pmatrix}.$$

Note that

$$\widetilde{M}_n(\alpha)=(2\cos\alpha)\cdot M_n\left(\frac{1}{2\cos\alpha}\right),$$

the formal difference is that $\widetilde{M}_n(\alpha)$ is also defined when $\cos\alpha=0$.

\begin{Proposition}\label{sin-formula}
The following identity holds:
$$\det(\widetilde{M}_n(\alpha))=\frac{\sin((n+1)\alpha)}{\sin\alpha}.$$

\end{Proposition}
\begin{proof}
It is trivial to check this identity for $n=0$ and $n=1$. By Lemma~\ref{det-recurrence} for $m=1$, the sequence $\{\det(\widetilde{M}_n(\alpha))\}_n$ satisfies the following recurrence:
$$\det(\widetilde{M}_n(\alpha))=2\cos\alpha\det(\widetilde{M}_{n-1}(\alpha))-\det(\widetilde{M}_{n-2}(\alpha)).$$

Thus, it is sufficient to prove that the sequence $\{\sin {n\alpha}\}_n$ satisfies the same recurrence. Indeed, we have:
\begin{align*}
\sin{n\alpha}&=\sin((n-1)\alpha)\cos\alpha+\cos((n-1)\alpha)\sin\alpha  \\
&=\sin((n-1)\alpha)\cos\alpha+\cos((n-2)\alpha)\cos\alpha\sin\alpha-\sin((n-2)\alpha)\sin^2\alpha\\
&=\sin((n-1)\alpha)\cos\alpha+\sin((n-1)\alpha)\cos\alpha-\sin((n-2)\alpha)\cos^2\alpha-\sin((n-2)\alpha)\sin^2\alpha\\
&=2\sin((n-1)\alpha)\cos\alpha-\sin((n-2)\alpha).
\end{align*}
Hence, we obtain the desired identity.
\end{proof}

\begin{Corollary}\label{roots}
The roots of the polynomial $P_n(x)$ are exactly the numbers
$$\frac{1}{2\cos\left(\frac{k\pi}{n+1}\right)},\ 1\le k\le n, \, k\neq\frac{n+1}2.$$
\end{Corollary}
\begin{proof}
Since $$\sin\left((n+1)\frac{k\pi}{n+1}\right)=\sin(k\pi)=0,$$
by Proposition~\ref{sin-formula}, one can see that all of these numbers are roots of $P_n(x)$. In addition, the degree of the polynomial $P_n(x)$ is at most $n$. The polynomial $P_n(x)$ is an even function since multiplying all odd rows and columns by -1 does not change the determinant. Therefore, for odd $n$, the polynomial $P_n(x)$ is of degree at most $n-1$.

This shows that the list of roots is complete.
\end{proof}

\begin{Theorem}\label{ML_odd_bound}
$$\MLdeg(C_{n})\ge 1+(n-3)\cdot 2^{n-2}. $$
\end{Theorem}
\begin{proof}
We consider two cases, depending on the parity of $n$:

\textbf{For the odd case $n=2m+1$:} By Corollary \ref{recurrence}(a), we have that
$$P_{2m-2}(x)=(P_{m-1}(x)-xP_{m-2}(x))(P_{m-1}(x)+xP_{m-2}(x)).$$

The polynomial $P_{2m-2}(x)$ has $2m-2$ distinct roots, by Corollary~\ref{roots}. Each of the factors is a polynomial of degree at most $m-1$, therefore, each of them must have $m-1$ distinct roots. In particular, the polynomial $P_{m-1}(x)+xP_{m-2}(x)$ has $m-1$ distinct roots. Hence, by Proposition \ref{ML-odd>1}, we conclude the result.

\textbf{For the even case $n=2m$:} By Corollary \ref {recurrence}(b), we have $$P_{2m-3}(x)=P_{m-2}(x)(P_{m-1}(x)-x^2P_{m-3}(x)).$$

By Corollary \ref{roots}, the polynomial $P_{2m-3}(x)$ has $2m-4$ distinct roots. Thus, the polynomials $P_{m-2}(x)$ and $P_{m-1}(x)-x^2P_{m-3}(x)$ have together $2m-4$ distinct roots. Hence, by Proposition \ref{ML-odd>1}, we conclude the result.
\end{proof}

\end{document}